\documentclass[11pt,a4paper]{amsart}
\usepackage[headings]{fullpage}
\usepackage{amssymb,amscd,hyperref}
\usepackage{amsthm}
\usepackage[alphabetic, nobysame]{amsrefs}
\usepackage[all]{xy}
\usepackage{color}
\usepackage{xcolor}
\usepackage{tikz-cd} 

\newcommand\blfootnote[1]{%
  \begingroup
  \renewcommand\thefootnote{}\footnote{#1}%
  \addtocounter{footnote}{-1}%
  \endgroup
}

\numberwithin{equation}{section}
\newtheorem{thm}{Theorem}[section]
\newtheorem{lem}[thm]{Lemma}
\newtheorem{prop}[thm]{Proposition}
\newtheorem{cor}[thm]{Corollary}

\theoremstyle{definition}

\newtheorem{ex}[thm]{Example}

\theoremstyle{remark}
\newtheorem{rem}[thm]{Remark}
\newtheorem*{acknowledgements}{Acknowledgements}

\def\Z{\mathbb{Z}}

\def\fix{\mathsf{fix}}

\def\norm{\mathsf{norm}}
\def\tran{\mathsf{tr}}
\def\trace{\mathsf{tr}}
\def\res{\mathsf{res}}
\def\mul{\mathsf{mult}}
\newcommand{\und}[1]{{\underline{#1}}}

\def\ra{\rightarrow}
\def\leq{\leqslant}
\def\geq{\geqslant}

\begin{document}
\title{Examples of \'etale extensions of Green functors}
\author{Ayelet Lindenstrauss}
\address{Mathematics Department, Indiana University, 831 East Third Street,
  Bloomington, IN 47405, USA}
\email{alindens@indiana.edu}
\author{Birgit Richter}

\address{Fachbereich Mathematik der Universit\"at Hamburg,
Bundesstra{\ss}e 55, 20146 Hamburg, Germany}
\email{birgit.richter@uni-hamburg.de}

\author{Foling Zou}
\address{Academy of Mathematics and Systems Science, Chinese Academy
of Science,  
No. 55 Zhongguancun East Road, Beijing 100190,  China}
\email{zoufoling@amss.ac.cn}
\date{\today}
\maketitle

\begin{abstract}
We provide new examples of \'etale extensions of Green functors by
transferring classical examples of \'etale extensions to the
equivariant setting. Our examples  are Tambara functors,
  and we prove Green \'etaleness for them, which implies Tambara
  \'etaleness. 
We show that every $C_2$-Galois extensions of
fields gives rise to an \'etale extension of $C_2$-Green
functors.
Here we associate the constant Tambara functor to the base
field and the fix-Tambara functor to the extension. We also prove that
all $C_n$-Kummer extensions give rise to \'etale extensions for
arbitrary finite $n$. \'Etale extensions of fields induce \'etale
extension of $G$-Green functors for any finite group $G$ by passing to
the corresponding constant $G$-Tambara functors. 

\end{abstract}

\section{Introduction}

In the world of commutative rings a finitely generated commutative
$k$-algebra $A$ is \'etale if it is flat and unramified, and the
latter property means that the module of K\"ahler differentials
$\Omega^1_{A|k}$ is trivial. Typical examples include localizations
and Galois extensions. Flat extensions with vanishing K\"ahler
differentials are often called formally \'etale. The module of
K\"ahler differentials can be identified with the first Hochschild
homology group of $A$ over $k$ and also with 
$I/I^2$, where $I$ is the kernel of the multiplication map $A
\otimes_k A \ra A$.  

In the equivariant setting, this notion turns out to be more involved:
Mike Hill defined a module of K\"ahler differentials for maps of
$G$-Tambara functors $\und{R} \ra \und{T}$ which we will denote by
$\Omega^{1,G}_{\und{T}|\und{R}}$ 
\cite[Definition 5.4]{hillaq} and he defined $\und{R} \ra \und{T}$ to
be formally \'etale, if $\und{T}$ is flat as an $\und{R}$-module and
if $\Omega^{1,G}_{\und{T}|\und{R}} =0$ \cite[Definition
5.9]{hillaq}. Flatness is defined analogously to the non-equivariant
context: $\und{T}$ is flat over $\und{R}$ if the Mackey box product
with $\und{T}$ over $\und{R}$, $\und{T} \Box_{\und{R}} (-)$, preserves
exactness. However, flatness is rarer in the equivariant context than
in the non-equivariant one, see \cite{hmq}. The definition of the
K\"ahler differentials is also slightly different: the basic object is
still the kernel of the multiplication map 
\[ \und{I} = \ker(\mul \colon \und{T} \Box_{\und{R}} \und{T} \ra \und{T})\]
but instead of quotienting out by $\und{I}^2$ one also kills the image
of norm maps that are induced by $2$-surjective maps of finite
$G$-sets. For a nice explicit description of
$\Omega^{1,G}_{\und{T}|\und{R}}$ see \cite[p.~38]{leeman}.  In all our
examples in this paper, we actually prove that $\und{I}^2=\und{I}$, so
we only use the corresponding underlying Green functor structures of
the Tambara functors involved. We still determine the full Tambara
structures of the relevant objects.

Examples of \'etale extensions are rare;  it is known that for 
a multiplicative subset
in $\und{R}$, the map $\und{R} \ra \und{R}[N^{-1}]$ is \'etale
\cite[Proposition 5.13]{hillaq}.  \blfootnote{MSC2020:
  Primary 55P91, Secondary 13B40}

The purpose of this note is to provide new families of examples of
\'etale extensions of Tambara functors. It grew out of our attempt to
understand the notion of \'etaleness in the equivariant context and to
transfer classical examples of \'etale extensions to the setting of
Tambara functors. We show that $C_2$-Galois extensions of fields $K
\subset L$ give rise to \'etale extensions of $C_2$-Tambara functors. Here
we associate the constant Tambara functor, $\und{K}^c$,  to the base
field and the fix-Tambara functor, $\und{L}^\fix$ to the extension. We
also prove that all $C_n$-Kummer extensions give rise to Tambara
\'etale extensions for arbitrary finite $n$.

\'Etale extensions of fields $K \subset L$ induce
\'etale extension of $G$-Tambara functors for any finite group
$G$. Here we consider the extension of the corresponding constant
$G$-Tambara functors, $\und{K}^c \ra \und{L}^c$.

In addition to sending a commutative $G$-ring $R$ to $\und{R}^c$ or
$\und{R}^\fix$, there is a third way of importing commutative rings
with a group action into the world of Tambara functors: There is a
left adjoint functor to the forgetful functor from $G$-Tambara
functors to $G$-rings. We investigate its properties in upcoming work
\cite{lrz}.  

It would be interesting to have
  an example of a Tambara \'etale extension which is not Green \'etale.
  We did not succeed in finding one.  
 
\begin{acknowledgements} We thank David Mehrle and Mike Hill for helpful
comments and encouragement. The authors also thank an anonymous referee
for comments which helped improve the exposition of this paper. 

The first named author was supported by NSF grant DMS-2004300 and a grant from the
Simons Foundation (917555, Lindenstrauss). The second named author thanks Churchill College
Cambridge for its hospitality. 

\end{acknowledgements}

\section{Preliminaries on Mackey functors and Tambara functors}
The box product is a symmetric monoidal product in the category of Mackey
functors. In this paper we use an explicit formula for box products of
$C_p$-Mackey functors due to \cite{lewis}.  
\begin{equation*}
(\und{M} \Box \und{N})(C_p/e) = \und{M}(C_p/e) \otimes \und{N}(C_p/e),
\end{equation*}
\begin{equation*}
(\und{M} \Box \und{N})(C_p/C_p) = \bigg(\und{M}(C_p/C_p) \otimes \und {N} (C_p/C_p) \oplus
\big(\und{M}(C_p/e) \otimes \und{N}(C_p/e)\big)_{\text{Weyl}}\bigg)/ \text{FR}.
\end{equation*}
Here,  the Weyl group, which is $C_p$, acts diagonally on $\und{M}(C_p/e) \otimes
\und{N}(C_p/e)$. The elements in the orbit $\big(\und{M}(C_p/e) \otimes \und{N}(C_p/e)\big)_{\text{Weyl}}$ are formal transfers and
    we will write their equivalence classes in $[-]$ to distinguish
    them from
    elements in the first part $\und{M}(C_p/C_p) \otimes \und {N} (C_p/C_p)$. The
    relation FR, short for Frobenius reciprocity, forces 
\begin{equation*}
x \otimes \tran(y) = [\res(x) \otimes y] \text{ and } \tran(x) \otimes y = [x \otimes \res(y)].
\end{equation*}

See  \cite[1.2.1]{mazur} for an explicit description of the box product structure for  the groups
$C_{p^n}$. 
For any commutative ring $R$ and prime $p$, we denote by $\und{R}^c$ the
constant $C_p$-Tambara functor with $\und{R}^c(C_p/e) = \und{R}^c(C_p/C_p) = R$
and $\norm(a) = a^p$, $\tran(a) = p\cdot a$ and $\res(a) = a$ for all $a \in R$. The
action of $C_p = \langle \tau\rangle$ on $\und{R}^c$ is trivial.

For a commutative $C_p$-ring $T$ we denote by $\und{T}^{\fix}$ the $C_p$-Tambara functor 
\[
  \und{T}^{\fix}(C_p/e) = T,  \qquad \und{T}^{\fix}(C_p/C_p) = T^{C_p},  
\]
Here, the restriction map $\res$ is the inclusion map $T^{C_p} \hookrightarrow T$, $\tran(a) =
a + \tau(a) + \cdots + \tau^{p-1}(a)$, and  $\norm(a) = a\tau(a)\cdots\tau^{p-1}(a)$.  

Note that the functor $(-)^{\fix}$ from the category of commutative rings with $C_p$-action to the category of $C_p$-Tambara functors is right adjoint to the evaluation functor 
  at the free orbit $C_p/e$. In particular, for every $C_p$-Galois
  extension $K \subset L$ the inclusion map $K=\und{K}^c(C_p/e) \to L$ is a $C_p$-map so it gives an adjoint map
  \begin{equation} \label{eq:adjoint} \und{K}^c \ra   \und{L}^{\fix}.\end{equation} 

\section{The case of $C_2$-Galois extensions of fields}
Let $K \subset L$ be a $C_2$-Galois extension of fields. 
Note that, depending on the characteristic of $K$, these can take one
of two forms: If the characteristic of $K$ is $2$, then a
$C_2$-extension is an Artin-Schreier extension \cite[Theorem 6.4]{lang}. If the characteristic
is prime to $2$, then $K \subset L$ is a Kummer extension. See
\cite[p.~89]{birch} for background. We will use the specific
forms of such extensions in  our arguments in both cases.

For any $C_2$-Galois extension $K\subset L$, we have a map of $C_2$-Tambara functors $\und{K}^c \ra
    \und{L}^{\fix}$ adjoint as in  \eqref{eq:adjoint} to the inclusion 
$K \ra L$ which is a $C_2$ map. In this section, we prove: 
\begin{thm} \label{thm:c2-flat}
 For a $C_2$-Galois extension $K \subset L$, the map $\und{K}^c \ra \und{L}^\fix$ is $C_2$-Tambara \'etale.
\end{thm}

The argument depends on the characteristic of the ground field $K$ and has two cases.
\subsection { Characteristic $2$:  Artin-Schreier Extensions
  }
  In this section, we will assume that
  $K$ is a field of characteristic $2$ and $L$ is a $C_2$-Galois extension of $K$.  Then $L = K(\alpha)$
 where $\alpha^2+\alpha + a=0$ for some $a \in K^\times$.
  \begin{lem}\label{LandL}
  If $K\subset L$ is a $C_2$-Galois extension and $K$ is a field of characteristic $2$ with  $L = K(\alpha)$ as above. Then the Tambara functor 
  $\und{L}^\fix \Box \und{L}^\fix$ is given by the diagram 
  \[ \xymatrix@R=0.5cm{
      \und{L}^\fix \Box \und{L}^\fix(C_2/C_2):   & K \otimes_{\Z} K \oplus \{[\lambda \alpha \otimes \mu \alpha], \lambda, \mu \in K\} \ar@/_3ex/[dd]_{\res} \\
      & \\
\und{L}^\fix \Box \und{L}^\fix(C_2/e):  & L \otimes_{\Z} L. \ar@<0.8ex>@/_3ex/[uu]_{\tran} \ar@<-0.8cm>@/_3ex/[uu]_{\norm}
    } \] 
  The restriction on $K\otimes_{\Z}K$ is the inclusion map, and
  \begin{align}
    \label{eq:1}
    \res[\lambda \alpha \otimes \mu \alpha] =  & \lambda \alpha \otimes \mu + \lambda \otimes \mu \alpha + \lambda \otimes \mu. 
 \end{align} 
 The transfer is given for $\lambda,\mu\in K$ by
 \begin{equation*}
\tran(\lambda \otimes \mu) = 0,\quad  \tran(\lambda \otimes \mu \alpha) = \lambda \otimes \mu = \tran(\lambda \alpha \otimes \mu) , \quad \tran(\lambda \alpha \otimes \mu \alpha) = [\lambda \alpha \otimes \mu \alpha].
\end{equation*}
The norm on pure
   tensors is calculated as the tensor product of the norms on generators with
 the help of the diagonal action;  the norm of a sum can be
   calculated by $\norm(x+y) = \norm(x) + \norm(y) + \tran(x \cdot \tau y)$ (see
   \cite[Example 1.4.1]{mazur} or, for the more general case of cyclic
   $p$-groups, \cite[Corollary 2.6]{hm}, \cite[section
   1.4.1]{mazur}).

  \end{lem} 
 
 \begin{proof}
The box product $\und{L}^\fix \Box \und{L}^\fix$ is given by the following diagram:
  \[ \xymatrix@R=0.5cm{
      \und{L}^\fix \Box \und{L}^\fix(C_2/C_2):   & (K \otimes_{\Z} K \oplus (L \otimes_{\Z}L)_{\text{Weyl}})/\text{FR} \ar@/_3ex/[dd]_{\res} \\
      & \\
\und{L}^\fix \Box \und{L}^\fix(C_2/e):  & L \otimes_{\Z} L. \ar@<0.8ex>@/_3ex/[uu]_{\tran} \ar@<-0.8cm>@/_3ex/[uu]_{\norm}
    } \] 
  As the tensor product is over the integers we have to consider elements $\lambda +
  \mu \alpha$ with $\lambda, \mu \in K$. Since the minimal polynomial
  of $\alpha$ over $K$ is $\alpha(\alpha + 1) = a$, $\tau(\alpha)=\alpha+1$ so the Weyl-action identifies $[\lambda \otimes \mu \alpha]$
  with
  \[[\lambda \otimes \mu(\alpha +1)] =[\lambda \otimes \mu \alpha + \lambda \otimes \mu]. \]
  Hence in the Weyl quotient all elements $[\lambda \otimes \mu]$ are trivial and we also identify $[\lambda \otimes \mu \alpha]$ with $[\lambda \otimes \mu(\alpha +1)]$. The dual result holds for $[\lambda \alpha \otimes \mu]$. For $[\lambda \alpha \otimes \mu \alpha]$ we obtain
  \[ [\lambda \alpha \otimes \mu \alpha] \sim [\lambda(\alpha+1) \otimes \mu(\alpha+1)]\]
  and this yields $[\lambda \alpha \otimes \mu] \sim [\lambda \otimes \mu \alpha]$.
  So we get Weyl equivalence classes $[\lambda \alpha \otimes \mu \alpha]$ and
  $[\lambda \otimes \mu \alpha] = [\lambda \alpha \otimes \mu]$. 

  The Frobenius reciprocity relations identify $[\lambda \otimes \mu \alpha]$ with
  \[ \lambda \otimes \trace(\mu \alpha) = \lambda \otimes \mu(\alpha + (\alpha +1)) = \lambda \otimes \mu \]
and  $[\lambda \alpha \otimes \mu]$ with $\lambda \otimes \mu$. So we are left with
   \[ \xymatrix@R=0.5cm{
      \und{L}^\fix \Box \und{L}^\fix(C_2/C_2):   & K \otimes_{\Z} K \oplus \{[\lambda \alpha \otimes \mu \alpha], \lambda, \mu \in K\} \ar@/_3ex/[dd]_{\res} \\
      & \\
\und{L}^\fix \Box \und{L}^\fix(C_2/e):  & L \otimes_{\Z} L, \ar@<0.8ex>@/_3ex/[uu]_{\tran} \ar@<-0.8cm>@/_3ex/[uu]_{\norm}
    } \] 
  as claimed in the lemma.  The restriction on $K\otimes_{\Z}K$ is  the inclusion map and to prove Equation (\ref{eq:1}), we observe that
  \begin{align*}
    \res[\lambda \alpha \otimes \mu \alpha] = & \res (\tran(\lambda \alpha \otimes \mu \alpha)) =  \trace(\lambda \alpha \otimes \mu \alpha)\\
\nonumber       = & \lambda \alpha \otimes \mu \alpha + \lambda(\alpha+1) \otimes \mu(\alpha +1) \\
\nonumber    = & \lambda \alpha \otimes \mu + \lambda \otimes \mu \alpha + \lambda \otimes \mu. 
 \end{align*} 
 The transfer is given by $\tran(a \otimes b) = [a \otimes b]$, so
 $\tran(\lambda \otimes \mu) = [\lambda \otimes \mu] = 0$,
 \begin{equation}
    \label{eq:2}
 \tran(\lambda \otimes \mu \alpha) = [\lambda \otimes \mu \alpha] \sim \lambda \otimes \mu \sim [\lambda  \alpha \otimes \mu] = \tran(\lambda \alpha \otimes \mu) 
\end{equation}
and $\tran(\lambda \alpha \otimes \mu \alpha) = [\lambda \alpha \otimes \mu \alpha]$. 
 \end{proof}
 
   \begin{lem}\label{LKL}
  If $K\subset L$ is a $C_2$-Galois extension and $K$ is a field of characteristic $2$ so $L = K(\alpha)$
 where $\alpha^2+\alpha + a=0$ for some $a \in K^\times$, then the Tambara functor 
  $  \und{L}^\fix \Box \und{K}^c \Box \und{L}^\fix$ is given by the diagram 
  \[ \xymatrix@R=0.5cm{
  \und{L}^\fix \Box \und{K}^c \Box \und{L}^\fix(C_2/C_2):   & K \otimes_{\Z} K \otimes_\Z K \oplus \{[\lambda \alpha \otimes \mu \otimes \nu \alpha], \lambda, \mu, \nu \in K\} \ar@/_3ex/[dd]_{\res} \\
      & \\
\und{L}^\fix \Box \und{K}^c \Box \und{L}^\fix(C_2/e):  & L \otimes_{\Z} K \otimes_\Z L, \ar@<0.8ex>@/_3ex/[uu]_{\tran} \ar@<-0.8cm>@/_3ex/[uu]_{\norm}
    } \] 
  where the restriction on $K \otimes_\Z K  \otimes_\Z K$ is the inclusion map and  for all $\lambda, \mu, \nu\in K$
\begin{align*}
\res[\lambda \alpha \otimes \mu \otimes \nu \alpha]  =& \lambda \alpha \otimes \mu \otimes \nu + \lambda \otimes \mu \otimes \nu \alpha + \lambda \otimes \mu \otimes \nu,
\end{align*}
  \begin{align*}
    \tran(\lambda \alpha \otimes \mu \otimes \nu \alpha) & = [\lambda \alpha \otimes \mu \otimes \nu \alpha], \\
    \tran(\lambda \otimes \mu \otimes \nu) & = 0, \\
    \tran(\lambda \alpha \otimes \mu \otimes \nu) & = \tran(\lambda \otimes \mu \otimes \nu \alpha) = \lambda \otimes \mu \otimes \nu.
\end{align*}
Norms are calculated on pure tensors coordinatewise and extended to sums as in the previous lemma.

\end{lem}

\begin{proof}
The definition of $\und{L}^\fix \Box \und{K}^c$ gives the diagram
  \[ \xymatrix@R=0.5cm{
      \und{L}^\fix \Box \und{K}^c(C_2/C_2):   & (K \otimes_{\Z} K \oplus (L \otimes_{\Z}K)_{\text{Weyl}})/\text{FR} \ar@/_3ex/[dd]_{\res} \\
      & \\
\und{L}^\fix \Box \und{K}^c(C_2/e):  & L \otimes_{\Z} K. \ar@<0.8ex>@/_3ex/[uu]_{\tran} \ar@<-0.8cm>@/_3ex/[uu]_{\norm}
    } \] 
 
  The Weyl action sends a generator $\lambda \alpha \otimes \mu$ to $\lambda(\alpha +1) \otimes \mu = \lambda \alpha \otimes \mu + \lambda \otimes \mu$. Hence $[\lambda \otimes \mu]$ is trivial for all $\lambda, \mu \in K$ and the class $[\lambda \alpha \otimes \mu]$ is equal to $[\lambda (\alpha +1) \otimes \mu]$.

  The only new relation in the Frobenius reciprocity identifies $\trace(\lambda_1 + \mu \alpha) \otimes \lambda_2$ with
\[  [\lambda_1 \otimes \lambda_2] + [\mu\alpha \otimes \lambda_2] = [\mu\alpha \otimes \lambda_2].\]  
As
\[ \trace(\lambda_1 + \mu\alpha) \otimes \lambda_2 = \mu \otimes \lambda_2\]
this identifies $[\mu \alpha \otimes \lambda_2]$ with $\mu \otimes \lambda_2$ and hence we just get
  \[ \xymatrix@R=0.5cm{
      \und{L}^\fix \Box \und{K}^c(C_2/C_2):   & K \otimes_{\Z} K  \ar@/_3ex/[dd]_{\res} \\
      & \\
\und{L}^\fix \Box \und{K}^c(C_2/e):  & L \otimes_{\Z} K \ar@<0.8ex>@/_3ex/[uu]_{\tran} \ar@<-0.4cm>@/_3ex/[uu]_{\norm}
    } \] 
  where the restriction map is the inclusion, $\tran(\lambda \alpha \otimes \mu) = [\lambda \alpha \otimes \mu] = \lambda \otimes \mu$, and $\tran(\lambda \otimes \mu) = 0$. The norm sends a generator $\lambda \otimes \mu$ to $\lambda^2 \otimes \mu^2$ and $\mu \alpha \otimes \nu$ to $\mu^2a \otimes \nu^2$.

  \bigskip
  The three-fold box product $\und{L}^\fix \Box \und{K}^c \Box \und{L}^\fix$ is
  \[ \xymatrix@R=0.5cm{
  \und{L}^\fix \Box \und{K}^c \Box \und{L}^\fix(C_2/C_2):   & (K \otimes_{\Z} K \otimes_\Z K \oplus (L \otimes_{\Z}K \otimes_\Z L)_{\text{Weyl}})/\text{FR} \ar@/_3ex/[dd]_{\res} \\
      & \\
\und{L}^\fix \Box \und{K}^c \Box \und{L}^\fix(C_2/e):  & L \otimes_{\Z} K \otimes_\Z L. \ar@<0.8ex>@/_3ex/[uu]_{\tran} \ar@<-0.8cm>@/_3ex/[uu]_{\norm}
    } \] 
 In $L \otimes_{\Z} K \otimes_\Z L$, as the Weyl action on
$\lambda \otimes \mu \otimes \nu \alpha$ gives $\lambda \otimes \mu \otimes \nu\alpha + \lambda \otimes \mu \otimes \nu$, we see that pure scalar
terms $\lambda \otimes \mu \otimes \nu$ are identified to zero.  That is, $[\lambda \otimes \mu \otimes
  \nu] = 0$. Modulo those, the action on 
  $\lambda \alpha \otimes \mu \otimes \nu \alpha$ yields
  \[ \lambda(\alpha + 1) \otimes \mu \otimes \nu(\alpha +1) =  \lambda \alpha \otimes \mu \otimes \nu \alpha + \lambda \alpha \otimes \mu \otimes \nu + \lambda  \otimes \mu \otimes \nu \alpha\]
and thus $[\lambda \alpha \otimes \mu \otimes \nu] = [\lambda \otimes \mu \otimes \nu \alpha]$. 
Again, the only new relation given by Frobenius reciprocity identifies $[\lambda \alpha \otimes \mu \otimes \nu]$ and $[\lambda \otimes \mu \otimes \nu \alpha]$ with $\lambda \otimes \mu \otimes \nu$. 
Thus we get
  \[ \xymatrix@R=0.5cm{
  \und{L}^\fix \Box \und{K}^c \Box \und{L}^\fix(C_2/C_2):   & K \otimes_{\Z} K \otimes_\Z K \oplus \{[\lambda \alpha \otimes \mu \otimes \nu \alpha], \lambda, \mu, \nu \in K\} \ar@/_3ex/[dd]_{\res} \\
      & \\
\und{L}^\fix \Box \und{K}^c \Box \und{L}^\fix(C_2/e):  & L \otimes_{\Z} K \otimes_\Z L. \ar@<0.8ex>@/_3ex/[uu]_{\tran} \ar@<-0.8cm>@/_3ex/[uu]_{\norm}
    } \] 
  The restriction on $K \otimes_\Z K  \otimes_\Z K$ is the inclusion map and 
\begin{align*}
\res[\lambda \alpha \otimes \mu \otimes \nu \alpha] = & \res (\tran(\lambda \alpha \otimes \mu \otimes \nu \alpha)) = \trace(\lambda \alpha \otimes
                             \mu \otimes \nu \alpha)\\
  =& \lambda \alpha \otimes \mu \otimes \nu + \lambda \otimes \mu \otimes \nu \alpha + \lambda \otimes \mu \otimes \nu.
\end{align*}
For the transfer, we obtain the formulas as given in the statement of the lemma.
\end{proof}

 \begin{prop} \label{prop:c2-AScase}
Let $K$ be a field of characteristic $2$ and $L$ be a $C_2$-Galois extension of $K$,
so $L = K(\alpha)$ where $\alpha^2+\alpha + a=0$ for some $a \in K^\times$. 
Then
 the relative box product $ \und{L}^\fix \Box_{\und{K}^c} \und{L}^\fix$ is given by
\begin{equation*}
\xymatrix@R=0.5cm{
  \und{L}^\fix \Box_{\und{K}^c} \und{L}^\fix(C_2/C_2):   & K \otimes_{K} K  \oplus \{[\lambda \alpha \otimes \alpha], \lambda \in K\} \ar@/_3ex/[dd]_{\res} \\
      & \\
\und{L}^\fix \Box_{\und{K}^c} \und{L}^\fix(C_2/e):  & L \otimes_K L. \ar@<0.8ex>@/_3ex/[uu]_{\tran} \ar@<-0.8cm>@/_3ex/[uu]_{\norm}}
\end{equation*}
with structure maps as specified in the proof below. Moreover, we have: 
\begin{enumerate}
\item The ideal $\und{I}$ at the fixed level, $\und{I}(C_2/C_2)$, is spanned by $1 \otimes 1 + [\alpha \otimes \alpha]$.
\item The $C_2$-Tambara K\"ahler differentials $\Omega_{\und{L}^\fix|\und{K}^c}^{1,C_2}$ vanish.
\end{enumerate}
\end{prop}

\begin{proof}
We use the results of Lemma \ref{LandL} and Lemma \ref{LKL} to construct the coequalizer $\und{L}^\fix \Box_{\und{K}^c} \und{L}^\fix$.
In the coequalizer, the $K$-action
on both copies of $L$ and on $K$ is identified and hence we obtain:
  \[ \xymatrix@R=0.5cm{
  \und{L}^\fix \Box_{\und{K}^c} \und{L}^\fix(C_2/C_2):   & K \otimes_{K} K  \oplus \{[\lambda \alpha \otimes \alpha], \lambda \in K\} \ar@/_3ex/[dd]_{\res} \\
      & \\
\und{L}^\fix \Box_{\und{K}^c} \und{L}^\fix(C_2/e):  & L \otimes_K L. \ar@<0.8ex>@/_3ex/[uu]_{\tran} \ar@<-0.8cm>@/_3ex/[uu]_{\norm}
    } \] 
For the rest of this proof, unadorned tensor products of elements
    are over $K$.  The restriction on $K \otimes_K K$ is the inclusion map and \begin{align*}
\res[\lambda \alpha \otimes \alpha] = & \lambda( \alpha \otimes 1 + 1 \otimes  \alpha + 1 \otimes 1),\\
    \tran(\lambda\cdot \alpha \otimes \alpha) & = [\lambda\cdot \alpha \otimes \alpha], \\
    \tran(\lambda\cdot 1 \otimes 1) & = 0, \\
  \tran(\lambda\cdot \alpha \otimes 1)& = \tran(\lambda\cdot 1 \otimes \alpha) = \lambda, \\
  \norm (\lambda\cdot \alpha \otimes \alpha) & = \lambda^2 a^2\\
  \norm(\lambda \cdot 1 \otimes 1 ) & = \lambda^2, \\
  \norm(\lambda\cdot \alpha \otimes 1) & = \norm (\lambda \cdot 1 \otimes \alpha) = \lambda^2 a.
\end{align*}

\item 
  The multiplication $\mul \colon \und{L}^\fix \Box_{\und{K}^c} \und{L}^\fix \ra \und{L}^\fix$ is a morphism of Tambara functors. Hence on $K \otimes_K K$ and $L \otimes_K L$ it is given by the multiplication on $K$ and $L$. As
  \begin{align*}
    \res (\mul[\lambda \alpha \otimes \alpha]) & =   \mul \circ \res[\lambda \alpha \otimes \alpha]\\
    & = \mul(\lambda(\alpha \otimes 1 + 1 \otimes \alpha + 1 \otimes 1))
    & \text{by\ Equation }~\eqref{eq:1} \\
    & = \lambda(2\alpha+1) = \lambda = \res(\lambda), 
  \end{align*}
  and the fact that restriction is injective in $\und{L}^{\fix}$,
  we get that $\mul[\lambda \alpha \otimes \alpha] = \lambda $. 
Therefore $1 \otimes 1 + [\alpha \otimes \alpha] \in\und{I}(C_2/C_2)$, the kernel of the fixed level of the
  multiplication map. By examination, $\und{I}(C_2/C_2)$  is spanned by $1 \otimes 1 +
  [\alpha \otimes \alpha] $.
  
\item 
The element $1 \otimes 1 + [\alpha \otimes \alpha] $ is actually
idempotent. For the following calculation we use the Frobenius
reciprocity formula for Green functors from \cite[p.~19]{mazur}.
\begin{align*}
  (1 \otimes 1 + [\alpha \otimes \alpha])(1 \otimes 1 + [\alpha \otimes \alpha]) & = 1 \otimes 1 +2[\alpha \otimes \alpha] + [\alpha \otimes \alpha][\alpha \otimes \alpha] \\
                                                                                     & = 1 \otimes 1 + [(\alpha \otimes \alpha) \res \, \tran(\alpha \otimes \alpha)] \\
                                                                                     & = 1 \otimes 1 + [(\alpha \otimes \alpha) \res(\alpha \otimes 1 + 1 \otimes \alpha + 1 \otimes 1)] \\
                                                                                     & = 1 \otimes 1 + [\alpha^2\otimes \alpha + \alpha \otimes \alpha^2 + \alpha \otimes \alpha] \\
                                                                                     & = 1 \otimes 1 + [(\alpha + a) \otimes \alpha + \alpha \otimes (\alpha + a) + \alpha \otimes \alpha] \\
                                                                                     & = 1 \otimes 1 + a[\alpha \otimes 1+1\otimes \alpha] + [\alpha \otimes \alpha] \\
 \text{ by\ Equation }~\eqref{eq:2}                      & = 1 \otimes 1 + 2a (1\otimes 1) + [\alpha \otimes \alpha]
                                         \\
  &= 1 \otimes 1 + [\alpha \otimes \alpha].
\end{align*}
This shows that $\und{I}/\und{I}^2(C_2/C_2) = 0$.
As $\und{I}(C_2/e)$ is just the kernel of the multiplication map $L \otimes_K L
  \ra L$ and as $K \subset L$ is \'etale, we know as well that $\und{I}/\und{I}^2(C_2/e) =
  0$. As
 \[ \Omega_{\und{L}^\fix|\und{K}^c}^{1,C_2} = \und{I}/\und{I}^{>1}\]
 and as $\und{I}/\und{I}^2$ maps surjectively onto $\und{I}/\und{I}^{>1}$, we
 obtain that $\Omega_{\und{L}^\fix|\und{K}^c}^{1,C_2}=0$. \qedhere
\end{proof}

\subsection { Characteristic $\neq 2$:  Kummer Extensions
  }
  In this section, we will assume that
  $K$ is a field of characteristic different than $2$ and $L$ is a $C_2$-Galois extension of $K$.  Then $L = K(\alpha)$
 where 
$\alpha^2=a \in K^\times$.
 
 \begin{lem}\label{LandLneq2}
If $K\subset L$ is a $C_2$-Galois extension and $K$ is a field of characteristic different than $2$, then  $L = K(\alpha)$ as above. The Tambara functor 
  $\und{L}^\fix \Box \und{L}^\fix$ is given by the diagram 

 \[ \xymatrix@R=0.5cm{
      \und{L}^\fix \Box \und{L}^\fix(C_2/C_2):   & K \otimes_{\Z} K \oplus \{[\lambda \alpha \otimes \mu \alpha], \lambda, \mu \in K\}\ar@/_3ex/[dd]_{\res} \\
      & \\
\und{L}^\fix \Box \und{L}^\fix(C_2/e):  & L \otimes_{\Z} L. \ar@<0.8ex>@/_3ex/[uu]_{\tran} \ar@<-0.8cm>@/_3ex/[uu]_{\norm}
    } \] 
  Restriction induces the inclusion on $K \otimes_{\Z} K$ and sends $[\lambda \alpha \otimes \mu \alpha]$ to $2\lambda \alpha \otimes \mu \alpha$. Transfer sends $\lambda \otimes \mu$ to $2 \lambda \otimes \mu$ and annihilates terms of the form $\lambda \otimes \mu \alpha$ and $\lambda \alpha \otimes \mu$. It sends $\lambda \alpha \otimes \mu \alpha$ to $[\lambda \alpha \otimes \mu \alpha]$. As $\norm(\alpha) = - \alpha^2 = -a$, we get that the norm map sends  $\lambda \alpha \otimes \mu$ to $- a \lambda^2 \otimes \mu^2$,
$\lambda \otimes \mu \alpha$ to $-  \lambda^2 \otimes a \mu^2$, and $\lambda \alpha \otimes \mu \alpha$ to $a \lambda^2 \otimes a \mu^2$. 
\end{lem}

 \begin{proof}
 Since the minimal polynomial of $\alpha$ over $K$ is $\alpha^2-a=0$, the Galois action sends $\alpha$ to $-\alpha$. The proof
  is analogous to the characteristic $2$ case.

For $\und{L}^\fix \Box \und{L}^\fix$ we get the following diagram:
  \[ \xymatrix@R=0.5cm{
      \und{L}^\fix \Box \und{L}^\fix(C_2/C_2):   & (K \otimes_{\Z} K \oplus (L \otimes_{\Z}L)_{\text{Weyl}})/\text{FR} \ar@/_3ex/[dd]_{\res} \\
      & \\
\und{L}^\fix \Box \und{L}^\fix(C_2/e):  & L \otimes_{\Z} L. \ar@<0.8ex>@/_3ex/[uu]_{\tran} \ar@<-0.8cm>@/_3ex/[uu]_{\norm}
    } \] 
  The Weyl relation identifies $[\lambda \alpha \otimes \mu \alpha]$ with itself and kills $[\lambda \alpha \otimes
  \mu]$ and $[\lambda \otimes \mu\alpha]$.  Frobenius reciprocity confirms $[\lambda \otimes \mu \alpha] =
  \lambda \otimes \trace(\mu\alpha) = 0$ 
  and also $[\lambda \alpha \otimes \mu] = 0$.  On scalars we obtain
\begin{equation}
\label{eq:3}
 [\lambda \otimes \mu] = \lambda \otimes \trace(\mu) =   \lambda \otimes 2 \mu = 2 \lambda \otimes \mu.
\end{equation}
  Elements like $[\lambda \alpha \otimes \mu \alpha]$ survive unharmed. So we have 
 \[ \xymatrix@R=0.5cm{
      \und{L}^\fix \Box \und{L}^\fix(C_2/C_2):   & K \otimes_{\Z} K \oplus \{[\lambda \alpha \otimes \mu \alpha], \lambda, \mu \in K\}\ar@/_3ex/[dd]_{\res} \\
      & \\
\und{L}^\fix \Box \und{L}^\fix(C_2/e):  & L \otimes_{\Z} L. \ar@<0.8ex>@/_3ex/[uu]_{\tran} \ar@<-0.8cm>@/_3ex/[uu]_{\norm}
    } \] 
  Restriction sends $[\lambda \alpha \otimes \mu \alpha]$ to $2\lambda \alpha \otimes \mu \alpha$. The transfer and norm calculations follow directly.
 \end{proof}
 
  \begin{lem}\label{LKLneq2}
 If $K\subset L$ is a $C_2$-Galois extension and $K$ is a field of characteristic different than $2$, so $L = K(\alpha)$
 where 
$\alpha^2=a \in K^\times$,  then the Tambara functor 
  $  \und{L}^\fix \Box \und{K}^c \Box \und{L}^\fix$ is given by the diagram 
 \[ \xymatrix@R=0.5cm{
  \und{L}^\fix \Box \und{K}^c \Box \und{L}^\fix(C_2/C_2):   & K \otimes_{\Z} K \otimes_\Z K \oplus \{[\lambda \alpha \otimes \mu \otimes \nu \alpha], \lambda,\mu,\nu \in K\} \ar@/_3ex/[dd]_{\res} \\
      & \\
\und{L}^\fix \Box \und{K}^c \Box \und{L}^\fix(C_2/e):  & L \otimes_{\Z} K \otimes_\Z L, \ar@<0.8ex>@/_3ex/[uu]_{\tran} \ar@<-0.8cm>@/_3ex/[uu]_{\norm}
} \]
where restriction is inclusion on the left part and $\res[\lambda \alpha \otimes \mu \otimes \nu\alpha] = 2\lambda \alpha \otimes \mu \otimes \nu\alpha$.  The transfer sends $\lambda \alpha \otimes \mu \otimes \nu\alpha$ to $[\lambda \alpha \otimes \mu \otimes \nu\alpha]$, $\lambda \otimes \mu \otimes \nu$ to
$2\lambda \otimes \mu \otimes \nu$, and the remaining terms to zero.  The norm is applied componentwise as before.
\end{lem}

\begin{proof}
We start with
$\und{L}^\fix \Box \und{K}^c$, given by
  \[ \xymatrix@R=0.5cm{
      \und{L}^\fix \Box \und{K}^c(C_2/C_2):   & (K \otimes_{\Z} K \oplus (L \otimes_{\Z}K)_{\text{Weyl}})/\text{FR} \ar@/_3ex/[dd]_{\res} \\
      & \\
\und{L}^\fix \Box \und{K}^c(C_2/e):  & L \otimes_{\Z} K. \ar@<0.8ex>@/_3ex/[uu]_{\tran} \ar@<-0.8cm>@/_3ex/[uu]_{\norm}
    } \] 
Again, the Weyl relation kills $[\lambda \alpha \otimes \mu]$ and Frobenius reciprocity identifies $[\lambda \otimes \mu]$ with $\lambda \otimes \trace(\mu) = 2\lambda \otimes \mu$. So we are left with 
  \[ \xymatrix@R=0.5cm{
      \und{L}^\fix \Box \und{K}^c(C_2/C_2):   & K \otimes_{\Z} K \ar@/_3ex/[dd]_{\res} \\
      & \\
\und{L}^\fix \Box \und{K}^c(C_2/e):  & L \otimes_{\Z} K \ar@<0.8ex>@/_3ex/[uu]_{\tran} \ar@<-0.4cm>@/_3ex/[uu]_{\norm}
    } \] 
  with restriction, transfers, and norms analogous to previous calculations.
  \bigskip
  For 
$\und{L}^\fix \Box \und{K}^c \Box \und{L}^\fix$ we have 
  \[ \xymatrix@R=0.5cm{
  \und{L}^\fix \Box \und{K}^c \Box \und{L}^\fix(C_2/C_2):   & (K \otimes_{\Z} K \otimes_\Z K \oplus (L \otimes_{\Z}K \otimes_\Z L)_{\text{Weyl}})/\text{FR} \ar@/_3ex/[dd]_{\res} \\
      & \\
\und{L}^\fix \Box \und{K}^c \Box \und{L}^\fix(C_2/e):  & L \otimes_{\Z} K \otimes_\Z L. \ar@<0.8ex>@/_3ex/[uu]_{\tran} \ar@<-0.8cm>@/_3ex/[uu]_{\norm}
} \]
We get that $[\lambda \alpha \otimes \mu \otimes \nu\alpha]$ survives, whereas
$[\lambda \alpha \otimes \mu \otimes \nu]=0=[\lambda \otimes \mu \otimes \nu\alpha]$. Again, pure scalars $[\lambda \otimes \mu \otimes \nu]$ are identified with $2\lambda\otimes \mu\otimes \nu$ via Frobenius reciprocity, so this three-fold box product is 
\[ \xymatrix@R=0.5cm{
  \und{L}^\fix \Box \und{K}^c \Box \und{L}^\fix(C_2/C_2):   & K \otimes_{\Z} K \otimes_\Z K \oplus \{[\lambda \alpha \otimes \mu \otimes \nu \alpha], \lambda,\mu,\nu \in K\} \ar@/_3ex/[dd]_{\res} \\
      & \\
\und{L}^\fix \Box \und{K}^c \Box \und{L}^\fix(C_2/e):  & L \otimes_{\Z} K \otimes_\Z L \ar@<0.8ex>@/_3ex/[uu]_{\tran} \ar@<-0.8cm>@/_3ex/[uu]_{\norm}
} \]
and $\tran$ sends $\lambda \alpha \otimes \mu \otimes \nu\alpha$ to $[\lambda \alpha \otimes \mu \otimes \nu\alpha]$ and $\lambda \otimes \mu \otimes \nu$ to
$2\lambda \otimes \mu \otimes \nu$. It sends the remaining terms to zero. Restriction is inclusion on the left part and $\res[\lambda \alpha \otimes \mu \otimes \nu\alpha] = 2\lambda \alpha \otimes \mu \otimes \nu\alpha$. The norm is applied componentwise.

\end{proof} 
\begin{prop}
  \label{prop:c2-Kummercase}
Let $K$ be a field of characteristic $ \neq 2$ and $L$ be a $C_2$-Galois extension of $K$, so  $L = K(\alpha)$ with
$\alpha^2=a \in K^\times$.
Then the relative box product $ \und{L}^\fix \Box_{\und{K}^c} \und{L}^\fix$ is
\begin{equation*}
\xymatrix@R=0.5cm{
  \und{L}^\fix \Box_{\und{K}^c} \und{L}^\fix(C_2/C_2):   & K \otimes_{K} K  \oplus \{[\lambda \alpha \otimes \alpha], \lambda \in K\} \ar@/_3ex/[dd]_{\res} \\
      & \\
\und{L}^\fix \Box_{\und{K}^c} \und{L}^\fix(C_2/e):  & L \otimes_K L, \ar@<0.8ex>@/_3ex/[uu]_{\tran} \ar@<-0.8cm>@/_3ex/[uu]_{\norm}}
\end{equation*}
with structure maps as specified in the proof below. Moreover, we have: 
\begin{enumerate}
\item The ideal $\und{I}$ at the fixed level, $\und{I}(C_2/C_2)$, is spanned by  $2a - [\alpha\otimes \alpha]$.
\item The $C_2$-Tambara K\"ahler differentials $\Omega_{\und{L}^\fix|\und{K}^c}^{1,C_2}$ vanish.
\end{enumerate}
\end{prop}
\begin{proof} We use the results of Lemma \ref{LandLneq2} and Lemma \ref{LKLneq2} to calculate the coequalizer $\und{L}^\fix \Box_{\und{K}^c} \und{L}^\fix$, to obtain
\[ \xymatrix@R=0.5cm{
  \und{L}^\fix \Box_{\und{K}^c} \und{L}^\fix(C_2/C_2):   & K \otimes_{K} K  \oplus \{[\lambda \alpha \otimes \alpha], \lambda \in K\} \ar@/_3ex/[dd]_{\res} \\
      & \\
\und{L}^\fix \Box_{\und{K}^c} \und{L}^\fix(C_2/e):  & L \otimes_K L \ar@<0.8ex>@/_3ex/[uu]_{\tran} \ar@<-0.8cm>@/_3ex/[uu]_{\norm}.
    } \] 
The restriction on $K \otimes_K K$ is the inclusion map and 
  \begin{align*}
\res[\lambda \alpha \otimes \alpha] = & 2\lambda \alpha \otimes  \alpha,\\
    \tran(\lambda\cdot \alpha \otimes \alpha) & = [\lambda\cdot \alpha \otimes \alpha], \\
    \tran(\lambda\cdot 1 \otimes 1) & = 2\lambda, \\
  \tran(\lambda\cdot \alpha \otimes 1)& = \tran(\lambda\cdot 1 \otimes \alpha) = 0, \\
  \norm (\lambda\cdot \alpha \otimes \alpha) & = \lambda^2 a^2\\
  \norm(\lambda \cdot 1 \otimes 1 ) & = \lambda^2, \\
  \norm(\lambda\cdot \alpha \otimes 1) & = \norm (\lambda \cdot 1 \otimes \alpha) = -\lambda^2 a.
\end{align*}

 The multiplication map $\mul \colon \und{L}^\fix \Box_{\und{K}^c}
  \und{L}^\fix \ra \und{L}^\fix$ induces the ordinary multiplication on $L$ at
  the $C_2/e$-level. As $\mul$ is a map of Tambara functors, we obtain that
  $\mul$ induces the  multiplication on $K \otimes_K K$. As $\res[\alpha \otimes \alpha] = 2\alpha \otimes \alpha$
  and as this is sent to $2\alpha^2= 2a$ under the multiplication map, we know that
  $\mul[\alpha \otimes \alpha] = 2a$. Therefore $2a - [\alpha\otimes \alpha]$ generates $\und{I}(C_2/C_2)$.

\item As  $\und{I}(C_2/e)$ is the kernel of the multiplication map $L \otimes_K L \ra
  L$
  and as $L$ is \'etale, we know that $\und{I}/\und{I}^2(C_2/e) =0$, and hence $\Omega^{1,C_2}_{\und{L}^\fix|\und{K}^c}(C_2/e)=0$.

\noindent   
Observe that 
  \[(2a-[\alpha \otimes \alpha])^2 = 4a^2-4a[\alpha \otimes \alpha] + [\alpha \otimes \alpha]^2 = 4a^2-4a[\alpha \otimes \alpha]+4a^2. \]
  Here, we use
  \[[\alpha \otimes \alpha]^2 = [\alpha \otimes \alpha \cdot
    \res\, \trace(\alpha \otimes \alpha)] = [\alpha^2 \otimes \alpha^2 + (-\alpha^2) \otimes (-\alpha^2)] = 2[a \otimes a] = 4a^2 \text{ by\ Equation }\eqref{eq:3}.\]
As the characteristic is not $2$ and as $a$ is invertible in $K$, dividing by $4a$ yields that $2a - [\alpha \otimes \alpha]$ is in $\und{I}^2(C_2/C_2)$. 
 Hence  $\Omega^{1,C_2}_{\und{L}^\fix|\und{K}^c}(C_2/C_2)=0$. \qedhere

\end{proof}

\subsection{Projectivity}
To prove \'etaleness, we still need projectivity.   Again, the proof depends on the characteristic of $K$.
\begin{lem} \label{lem:proj2}  (Artin-Schreier case)
For any $C_2$-Galois extension $K \subset L$ where the characteristic of $K$ is $2$, the $C_2$-Tambara functor $\und{L}^\fix$ is projective over $\und{K}^c$. 
\end{lem}  
\begin{proof}
The characteristic of $K$ is $2$, so we know that $L= K(\alpha)$ with $\alpha^2+\alpha+a=0$ for some $a \in K$ and $\tau(\alpha) = \alpha +1$. Assume that 
$\pi \colon \und{M} \ra \und{Q}$ is an epimorphism of $C_2$-Mackey functors that are $\und{K}^c$-modules. In particular, the values on the orbits are $K$-vector spaces. Let $\zeta \colon \und{L}^\fix \ra \und{Q}$ be a morphism of $\und{K}^c$-modules. Consider the diagram
\[\xymatrix{
K \ar@/^5ex/[drr]^{\zeta(C_2/C_2)} & & \\
& \und{M}(C_2/C_2) \ar[r]^{\pi(C_2/C_2)}\ar@<-0.5cm>[d]_\res& \und{Q}(C_2/C_2) \ar@<-0.5cm>[d]_\res\\ 
& \und{M}(C_2/e) \ar[r]^{\pi(C_2/e)}  \ar@<-0.5cm>[u]_\tran& \und{Q}(C_2/e)\ar@<-0.5cm>[u]_\tran\\ 
L \ar@/_5ex/[urr]_{\zeta(C_2/e)}&&
  } \]
We define $\xi\colon \und{L}^\fix \ra \und{M}$ as follows: We set $\tilde{m} := \zeta(C_2/e)(\alpha) \in \und{Q}(C_2/e)$ and choose a preimage $m$ of $\tilde{m}$ under $\pi(C_2/e)$. We define $\xi(C_2/e)(\alpha):= m$. Then $1 = \alpha + \tau \alpha$ is sent to
\[ \xi(C_2/e)(\alpha + \tau \alpha) = m + \tau m. \]

We define $\xi(C_2/C_2)(1) := \tran(m)$. Then
\[ \res(\xi(C_2/C_2)(1)) = \res \, \tran(m) = m + \tau m = \xi(C_2/e)(1) = \xi(C_2/e)(\res(1))\]
and by construction $\xi$ commutes with $\tran$. 
\end{proof}

\begin{lem} \label{lem:projneq2}   (Kummer case)
For any $C_2$-Galois extension $K \subset L$ where the characteristic of $K$ is not $2$, the $C_2$-Tambara functor $\und{L}^\fix$ is projective over $\und{K}^c$. 
\end{lem}  
\begin{proof}

  Since the characteristic of $K$ is not equal to $2$, $\tran(1) = 2$ is invertible. We have $L = K(\alpha)$ with $\alpha^2=a\in K$ and $\tau(\alpha)=-\alpha$. In particular, the transfer on $\alpha$ is zero.

  We claim that any $\und{K}^c$-module $\und{M}$ can be canonically decomposed
as  $\und{M} = \und{M}^+ \oplus \und{M}^-$  in the following way.
  As $2 \in K^\times$ we can split all $C_2$-representations into the
  $\pm$-eigenspaces of the $\tau$-action via the usual trick of writing $x =
  \frac{x+\tau x}{2} + \frac{x -\tau x}{2}$.
So we can write $\und{M}(C_2/e) = \und{M}^{+}(C_2/e) \oplus
    \und{M}^-(C_2/e)$.
    For $x \in \und{M}^{-}(C_2/e)$, we have  $\tran(x) = \tran(\tau x) = - \tran(x)$ and
    hence $\tran(x) = 0$.
    Define $\und{M}^+(C_2/C_2) = 
    \und{M}(C_2/C_2)$ and $\und{M}^-(C_2/C_2) = 0$.  Then $\und{M}^+$ and
    $\und{M}^-$ are sub-$C_2$-Mackey-functors of $\und{M}$ and $\und{K}^c$-modules. 

Our lifting diagram looks as follows:

\[\xymatrix{
K = K \oplus 0 \ar@/^5ex/[drr]^{\zeta(C_2/C_2)} & & \\
& \und{M}^+(C_2/C_2) \oplus \und{M}^-(C_2/C_2) \ar[r]^{\pi(C_2/C_2)}\ar@<-0.5cm>[d]_\res& \und{Q}^+(C_2/C_2) \oplus \und{Q}^-(C_2/C_2) \ar@<-0.5cm>[d]_\res\\ 
& \und{M}^+(C_2/e)  \oplus \und{M}^-(C_2/e) \ar[r]^{\pi(C_2/e)}  \ar@<-0.5cm>[u]_\tran& \und{Q}^+(C_2/e)  \oplus \und{Q}^-(C_2/e)\ar@<-0.5cm>[u]_\tran\\ 
K \oplus \{\lambda\alpha, \lambda \in K\} \ar@/_5ex/[urr]_{\zeta(C_2/e)}&&
  } \]
 It suffices to show that both $(\und{L}^{\fix})^+$ and
  $(\und{L}^{\fix})^-$ are projective. 
    For the positive part, $(\und{L}^\fix)^+
  \cong \und{K}^c$ is a free $\und{K}^c$-module, so in
particular projective. 
For the negative part, the only nontrivial module is at the $C_2/e$-level, and as $(\und{L}^{\fix})^-(C_2/e) = K$ is projective, we can solve the lifting problem.
\qedhere
\end{proof}

\begin{proof}[Proof of Theorem~\ref{thm:c2-flat}]
  In Propositions~\ref{prop:c2-AScase} and \ref{prop:c2-Kummercase}, we proved that
  $\und{K}^c \ra \und{L}^\fix$ has vanishing K\"ahler differentials.
In Lemma~\ref{lem:proj2} and Lemma~\ref{lem:projneq2}, we proved that $\und{K}^c \ra \und{L}^\fix$ is
projective.
As projective $\und{K}^c$-modules are flat (see \cite[Corollary 6.4]{LewisPiF}
or \cite[Prop 2.2.13]{leeman}),
we know that $\und{K}^c \ra \und{L}^\fix$ is flat.
\end{proof}
\section{Tambara \'etaleness of $C_n$-Kummer extensions}

By the above discussion we already know that $C_2$-Kummer extensions
are Tambara \'etale. We generalize this to arbitrary finite cyclic
groups: Let $K \subset L$ be a $C_n$-Kummer extension
\cite[p.~89]{birch}, hence $n$ is invertible in $K$ and the polynomial
$X^n-1$ splits in $K$. We denote by $\zeta_n$ a primitive
$n$th root of unity. We can assume that $L$ is of the form $L =
K(\alpha)$ with  $\alpha^n =a \in K^{\times}$, and
that the generator $\sigma$ of the Galois group $C_n$
  acts on $\alpha$ via multiplication by $\zeta_n$.

We start with a result that helps to prove flatness. The proof generalizes the one of Lemma
\ref{lem:projneq2}.

\begin{lem}
  \label{lem:decompose}
 Let $K$ be a field containing a primitive $n$th root of unity $\zeta
 = \zeta_n$ with $n$ invertible in $K$.
  Then any
 $\und{K}^c$-module $\und{M}$ can be decomposed uniquely and functorially as $\bigoplus_{i=0}^{n-1}
 \und{M}^{\zeta^i}$, characterized by the following properties that for any subgroup $C_m \leq C_n$,
\begin{enumerate}
\item the generator of $C_n/C_m$ acts on
  $\und{M}^{\zeta^i}(C_n/C_m)$ as multiplication by $\zeta_n^i$;
\item  if $m\nmid i$,  then $\und{M}^{\zeta^i}(C_n/C_m) = 0$;
\item if $m\mid i$, then both $\res_e^{C_m}$ and $\tran_{e}^{C_m}$ in
  $ \und{M}^{\zeta^i}$  are isomorphisms.
\end{enumerate}
\end{lem}
\begin{proof}
(1) Any $C_n$-module $M$ can be decomposed into the
$\zeta_n^i$-eigenspaces  $\bigoplus_{i=0}^{n-1} M^{\zeta_n^i}$ using the identity
\[ x = \sum_{i = 0}^{n-1}
  \frac{\sum_{j=0}^{n-1}\zeta_n^{-ij}\sigma^{j}(x)}{n}. \]
We consider $\und{M}(C_n/C_m)$ with the $C_n$-action via $C_n \to
C_n/C_m = W_{C_n}(C_m)$. In this way, the restriction maps and the transfer maps are
$C_n$-equivariant. Setting  
$\und{M}^{\zeta^i}(C_n/C_m) = \und{M}(C_n/C_m)^{\zeta_{n}^i}$, we obtain
sub-Mackey functors $\und{M}^{\zeta^i}$ of $\und{M}$.

(2) As $\und{M}(C_n/C_m)$
is $C_m$-fixed, we obtain  $\und{M}^{\zeta^i}(C_n/C_m) = 0$ if $\zeta_{n}^i$ is not of the
form of $\zeta_{n/m}^j$ for some $j$, or equivalently, if $ m \nmid i$.

(3) For a $\und{K}^c$-Mackey functor,   $\tran_{e}^{C_m} \circ
\res_e^{C_m} = m$ (as noticed by \cite[Remark 2.13]{zeng}). This holds because
by Frobenius reciprocity, $\tran_{e}^{C_m} (\res_e^{C_m}(x)) =  \tran_{e}^{C_m}
(1) \cdot x = m \cdot x$.
For $m \mid i$ we get that
\[\res_e^{C_m}(\tran_{e}^{C_m}(x)) = (1+(\zeta_n^i)^{n/m} +
  (\zeta_n^i)^{2n/m} + \cdots + (\zeta_n^i)^{(m-1)n/m})\cdot x = m \cdot x.\]
As $n$  is invertible in $K$,  so is $m$.  This shows that $\tran$ and
  $\res$ are inverses of each other up to multiplication by a unit. 
\end{proof}

\begin{ex} \label{ex:Lfix-decompose}
  For a $C_n$-Kummer extension $K \subset L = K(\alpha)$ with  $\alpha^n =a \in K^{\times}$, for $0 \leq i < n$,
  \[(\und{L}^{\fix})^{\zeta^i} (C_n/C_m)
  =\begin{cases}
     K\{\alpha^{i}\} & m \mid i\\
     0 & m \nmid i.
   \end{cases} \]
 For $d \mid m \mid i$, $\res^{C_m}_{C_d}(\alpha^i) = \alpha^i$; for $d \mid i$ and $d \mid m$,  
 \[ 
   \tran_{C_d}^{C_m}(\alpha^i) = \begin{cases}
                               \frac{m}{d} \alpha^i & m \mid i \\
                               0 & m \nmid i.
\end{cases}\]
\end{ex}

\begin{cor}
  \label{cor:decompose}
 Let $K$ be a field containing a primitive $n$th root of unity $\zeta_n$ where $n$
 is invertible in $K$. Then: 
\begin{enumerate}
\item  Evaluation
 at the $C_n/e$-level gives an  equivalence 
between the category of 
$\und{K}^c$-modules in $C_n$-Mackey functors and the category of
$K[C_n]$-modules.
\item All $\und{K}^c$-modules are projective. 
\end{enumerate}
\end{cor}
\begin{proof}
(1)  By Lemma~\ref{lem:decompose}, there are natural isomorphisms of
  $\und{K}^c$-modules $\und{M} \cong (\und{M}(C_n/e))^{\fix}$ for all $\und{M}$.
 Lemma~\ref{lem:decompose} implies that
  this isomorphism is defined to be the identity at the $C_n/e$-level and propagates
  to the other levels as the relevant non-trivial structure maps are all 
  isomorphisms.

(2) Let $\und{M}$ be a $\und{K}^c$-module. By the previous part, in order to
  show that it is projective, it suffices to show that $\und{M}(C_n/e)$ is a
  projective $K[C_n]$-module. But from the assumption on $K$, all
  $K[C_n]$-modules split as sums of one-dimensional ones, which are projective. 
\end{proof}

\begin{thm}
  \label{sec:tamb-etal-c_n}
  Let $K\subset L$ be a $C_n$-Kummer extension. Then $\und{K}^{c} \to \und{L}^{\fix}$ is $C_n$-Tambara \'etale.
\end{thm}
\begin{proof}
Recall that $L = K(\alpha)$,  $\alpha^n =a \in K^{\times}$, $n \in K^\times$ and that the generator $\sigma$ of $C_n$
  acts on $\alpha$ via multiplication by $\zeta_n$, an $n$th root of unity.
  We 
  show that all elements in
  $\und{I}: = \ker(\mul \colon \und{L}^{\fix} \square_{\und{K}^c} \und{L}^{\fix} \to \und{L}^{\fix})$ are also in $\und{I}^2$. Note that when $n$ is not prime, there are 
  orbits other than the trivial one and the free one, and there are
  also transfer elements from there.

  In the following we use the
  formulas for the levels of a box product from \cite[Definition
  1.2.1]{mazur}. For a subgroup $C_m \leq C_n$ we have $L^{C_m} =  K(\alpha^m)$, and
\begin{align*}
  (\und{L}^{\fix} \square_{\und{K}^c}  \und{L}^{\fix})(C_n/C_m)
  & = K(\alpha^m) \otimes _K K(\alpha^m) \oplus (( \bigoplus_{d \mid m} K(\alpha^d) \otimes_K K(\alpha^d))_{\mathrm{Weyl}})/ \mathrm{FR}
\end{align*}
We will first determine $\und{I}(C_n/C_m)$.
Let $[x]_d^m$ denote the transfer class of $x$ from level $C_n/C_d$ to $C_n/C_m$ in
the box product and let $C_1 := e$.
\begin{itemize}
\item The kernel of the multiplication map
  $I_m:= \ker(\mul \colon K(\alpha^m) \otimes _K K(\alpha^m) \to K(\alpha^m)) $ is contained
in $\und{I}(C_n/C_m)$.
\item Let $q: = m/d$. 
  Note that $L^{C_d} \otimes_K L^{C_d}$ has a $C_m/C_d$-action, and the generator acts
  by  
\begin{equation*}
\sigma^{n/m}(\alpha^d) = (\zeta_n^{n/m})^{d}\cdot \alpha^d = \zeta_{m/d} \cdot \alpha^d = \zeta_q \cdot \alpha^d.
\end{equation*}
 
\begin{itemize}
\item  If $x$ is not  fixed  in $L^{C_d} \otimes_K L^{C_d}$, then 
\begin{equation*}
 \res^{C_m}_{C_d}([x]_d^m) = (1+ \zeta_q + \zeta_{q}^2 + \cdots + \zeta_{q}^{q-1} )x = 0.
\end{equation*}
Any element in the kernel of the restriction map to $C_n/e$
is automatically in $\und{I}(C_n/C_m)$  because the restriction map in
$\und{L}^{\fix}$ is injective.
\item If $x$ is fixed in $L^{C_d} \otimes_K L^{C_d}$, then $\res([x]_{d}^{m}) = q
  \cdot x \neq 0$. This occurs if $x$ is a $K$-linear combination of elements of the form $\alpha^{id} \otimes
  \alpha^{jd}$ for $q\mid(i+j)$. In these cases  
\[q \otimes \alpha^{(i+j)d} - [\alpha^{id} \otimes  \alpha^{jd}]_d^m \in  \und{I}(C_n/C_m). \]
Note that this makes sense because $m=qd$ divides $(i+j)d$. 
\end{itemize}
\end{itemize}

Therefore, 
$\und{I}(C_n/C_m)$ contains $I_m$ and $\ker(\res)$ and its other generators
are  of the form 
\begin{equation}
\label{eq:4}
q \otimes \alpha^{(i+j)d} - [\alpha^{id} \otimes  \alpha^{jd}]_d^m \text{ for } m=qd \text{ and } q\mid(i+j).
\end{equation}

\medskip
We need to prove that  $\Omega^{1,C_n}_{\und{L}^{\fix}
  |\und{K}^c}$ is trivial when evaluated on all $C_n$-sets $C_n/C_m$
for all subgroups $C_m \leq C_n$. We prove, as before,  that
  $\und{I}(C_n/C_m) = \und{I}^2(C_n/C_m)$.

First, as $K \subset K(\alpha^m)$ is a $C_{n/m}$-Kummer extension and has vanishing K\"ahler
differentials, we obtain that  $I_m =  I_m^2 \subset \und{I}^2(C_n/C_m)$.

Next, we show that if $z \in \und{I}(C_n/C_m)$ restricts to $0 \in
\und{I}(C_n/e)$,  then $z$ is in $\und{I}^2$.
This is true because the element $m \otimes \alpha^n - [\alpha \otimes \alpha^{n-1}]_1^m = ma-[\alpha \otimes \alpha^{n-1}]_1^m $
is in $\und{I}(C_n/C_m)$, and 
\begin{align*}
  z \cdot (m a - [\alpha \otimes \alpha^{n-1}]_1^m)  = ma\cdot z - z[\alpha \otimes \alpha^{n-1}]_1^m = ma\cdot z - [
  \res(z) \cdot \alpha \otimes \alpha^{n-1}]_1^m =  ma\cdot z
\end{align*}
agrees with $z$ up to a unit.

We are left with elements of the form $x_{i,i+j}: = q \otimes \alpha^{(i+j)d} - [\alpha^{id} \otimes
\alpha^{jd}]_d^m$ as in \eqref{eq:4}. Here, $q\mid(i+j)$ and $1 < i,j < n/d$, but we can  allow all $i,j \geq 0$, because of the  relations 
\[\begin{array}{rll}
  x_{0,t} = x_{t,t} & = 0 & \text{ by Frobenius reciprocity } \\
  x_{i+n/d,t} & = a x_{i,t} & \text{ when }i+n/d \leq t \\
  x_{i+n/d,t+n/d} & = a x_{i,t}. &  
\end{array}\]
Consider the element $1 \otimes \alpha^m - \alpha^m \otimes 1 $ in $I_m$. Then 
\begin{align*}
  &(1 \otimes \alpha^m - \alpha^m \otimes 1) \cdot x_{i,i+j} \\
  & = (1 \otimes \alpha^{qd} - \alpha^{qd} \otimes 1)\cdot(q \otimes \alpha^{(i+j)d} - [\alpha^{id} \otimes
  \alpha^{jd}]_d^m) \\
&  =  q \otimes \alpha^{(i+j+q)d} - [\alpha^{id} \otimes
\alpha^{(j+q)d}]_d^m - q \alpha^{qd} \otimes \alpha^{(i+j)d} - [\alpha^{(i+q)d} \otimes
                \alpha^{jd}]_d^m \\
&   = x_{i,i+j+q} - x_{i+q,i+j+q} +  q \otimes \alpha^{(i+j+q)d}   - q \alpha^{qd} \otimes \alpha^{(i+j)d}.
\end{align*}
As $q \otimes \alpha^{(i+j+q)d}   - q \alpha^{qd} \otimes \alpha^{(i+j)d}$ is in $I_m$ and thus in $I_m^2$, we get 
\begin{equation}
\label{eq:5}
x_{i,i+j+q} \equiv x_{i+q,i+j+q} \text{ mod } \und{I}^2.
\end{equation}
A similar computation shows 
$    x_{i,t} \cdot x_{i',t'} = x_{i,t+t'} + x_{i', t+t'} - x_{i+i',t+t'},$
so
\begin{equation}
\label{eq:6}
 x_{i,t} \equiv i \cdot x_{1,t} \text{ mod } \und{I}^2.
\end{equation}
Combining \eqref{eq:5} and \eqref{eq:6}, we have
\begin{equation*}
x_{0,t} \equiv x_{q,t} \equiv q \cdot x_{1,t} \text{ mod } \und{I}^2.
\end{equation*}
Now, $x_{0,t} = 0$ and $q$, being a factor of $n$, is invertible.
This proves that $x_{1,t}$ and thus $x_{i,t}$ is in $\und{I}^2$.
As $m,d$ are arbitrary in the proof, we have shown 
$\Omega^{1,C_n}_{\und{L}^{\fix} |\und{K}^c} = 0$.

  Projectivity of $\und{L}^\fix$ over $\und{K}^c$ can be shown again via the
    decomposition into eigenspaces,  as shown in
      Corollary~\ref{cor:decompose}. 

\end{proof}

\begin{rem}
  In all our examples the K\"ahler differentials vanish because the
  generators of $\und{I}$ can be shown to be elements of
  $\und{I}^2$. So all our examples are \'etale extensions of the
  underlying Green functors. 

  We could however in the $C_3$-case for instance kill $[\alpha^2 \otimes \alpha] - [\alpha \otimes \alpha^2]$ with a norm, as
  \[ \norm(1 \otimes \alpha -\alpha \otimes 1) =  [\alpha^2 \otimes \alpha] - [\alpha \otimes \alpha^2]. \] 

  For the $C_2$-Galois examples
  $1 \otimes 1 + [\alpha \otimes \alpha]$ can also be written as $\norm(1 \otimes \alpha + \alpha \otimes 1)$ in
  characteristic $2$, so it has two reasons to die. Similarly, in the
  $C_2$-Kummer case $2a -[\alpha \otimes \alpha] = \norm(1 \otimes \alpha - \alpha \otimes 1)$.
  
\medskip 
  One could of course hope that for any finite $G$-Galois extension $K \subset L$, the extension $\und{K}^c \ra \und{L}^\fix$ is $G$-Tambara \'etale, but such a claim would need a more conceptual proof than chasing generators of $\und{I}$ to their death. 
  \end{rem}

\section{Constant Tambara functors on \'etale extensions}
We assume that $K$ is a field of arbitrary characteristic and that $K \subset L$ is an \'etale extension. In particular, the underlying $K$-vector space of $L$ is finite dimensional. The aim of this section is to prove that $\und{K}^c \ra \und{L}^c$ is a $G$-Tambara extension for every finite group $G$. To this end we provide a proof of a probably well-known identification of box-products of constant Tambara functors:

\begin{lem}
  Assume that $A$ and $B$ are commutative rings. Then
 $\und{A}^c \Box \und{B}^c \cong \und{(A \otimes B)}^c$ as $G$-Tambara functors for every finite group $G$. 
\end{lem}

\begin{proof}

We use that $\und{A}^c \Box \und{B}^c$ is the coproduct of $\und{A}^c$ and $\und{B}^c$ in the category of $G$-Tambara functors \cite[Lemma 9.8]{strickland}. 

  Let $\und{R}$ be an arbitrary $G$-Tambara functor and let $\varphi \colon \und{A}^c \ra \und{R}$ and $\psi \colon \und{B}^c \ra \und{R}$ be morphisms of $G$-Tambara functors. The unique induced morphism $\xi \colon \und{A}^c \Box \und{B}^c \ra \und{R}$ is determined by Weyl equivariant ring morphisms 
  \[ \xi(G/H) \colon A \otimes B = \und{A}^c(G/H) \otimes \und{B}^c(G/H) \ra \und{R}(G/H)\]
  for all $H < G$ 
  that satisfy compatibility constraints with respect to restriction, $\tran$ and $\norm$ coming from the definition of the box product as a Day convolution product. See \cite[Remark 1.2.3 and p.~41]{mazur} for an explicit list of requirements.
  
As the tensor product is the coproduct in the category of commutative rings, $\xi(G/H)$ is uniquely determined by ring maps $\xi_A(G/H) \colon A \ra \und{R}(G/H)$ and $\xi_B(G/H) \colon B \ra \und{R}(G/H)$. 

  As the action on the constant Tambara functors is trivial, the image of the $\xi(G/H)$ is contained in the Weyl fixed points. As restriction on $\und{A}^c$ and $\und{B}^c$ is the identity map, we get that for $H < K$, $\res_H^K \circ \xi(G/H) = \xi(G/K)$. The compatibility with $\tran$ demands that multiplication by the index in one tensor factor first and then applying $\xi$ is the same as first applying $\xi$ and then applying $\tran$ in $\und{R}$. Similarly, applying first the norm in both factors and then applying $\xi$ has to agree with first applying $\xi$ and then applying the norm in $\und{R}$.  These are exactly the requirements for obtaining a morphism of $G$-Tambara functors from $\und{(A\otimes B)}^c$ to $\und{R}$.

\end{proof}
\begin{cor}
For every \'etale extension $K \subset L$ where $K$ is a field, the extension $\und{K}^c \ra \und{L}^c$ is $G$-Tambara \'etale for every finite group $G$. 
\end{cor}  
\begin{proof}
  As $\und{K}^c$-modules in $G$-Mackey functors $\und{L}^c$ splits into a direct sum of $\und{K}^c$s and hence is free, thus flat.

  The coequalizer $\und{L}^c \Box_{\und{K}^c} \und{L}^c$ of
  $\xymatrix@1{\und{L}^c \Box \und{K}^c\Box \und{L}^c \ar@<0.5ex>[r] \ar@<-0.5ex>[r] &   \und{L}^c \Box \und{L}^c}$ 
  is isomorphic to the coequalizer of
  $\xymatrix@1{\und{(L \otimes K \otimes L)}^c \ar@<0.5ex>[r] \ar@<-0.5ex>[r] &   \und{(L \otimes L)}^c}$ 
  and this is nothing but $\und{(L \otimes_K L)}^c$. The multiplication map
  \[ \und{L}^c \Box_{\und{K}^c} \und{L}^c \ra \und{L}^c\]
  is induced by the multiplication map on $L$ over $K$ and hence its kernel is $\und{I}^c$ where $I$ denotes the kernel of the multiplication map on $L$ over $K$. As $K \subset L$ is \'etale, $I/I^2=0$, and hence $\Omega^{1,G}_{\und{L}^c|\und{K}^c} =0$.   
\end{proof}

\begin{rem}
This result can be generalized to \'etale extensions of rings $R \ra A$ as long as $\und{A}$ is a flat $\und{R}$-module. 
\end{rem}


\begin{bibdiv}
  \begin{biblist}
    


    \bib{birch}{incollection}{
      author={Birch, Bryan John},
      title={Cyclotomic fields and Kummer extensions}, 
     booktitle    = {Algebraic Number Theory, Cassels, J.~W.~S. and Fr\"ohlich, A.},
  year         = {2010}, 
 pages        = {85--93},
publisher    = {London Mathematical Society},
  note         = {Second Edition},

    }    




    
\bib{hillaq}{article}{
AUTHOR = {Hill, Michael A.},
     TITLE = {On the {A}ndr\'{e}-{Q}uillen homology of {T}ambara functors},
   JOURNAL = {J. Algebra},
     VOLUME = {489},
      YEAR = {2017},
     PAGES = {115--137},
}

 \bib{hm}{article}{
    AUTHOR={Hill, Michael A.},
    AUTHOR={Mazur, Kristen},
TITLE={An equivariant tensor product on Mackey functors}, 
JOURNAL={J. Pure Appl. Algebra},
YEAR={2019},
ISSUE= {12},
PAGES={5310--5345},
}
  
  \bib{hmq}{article}{
    AUTHOR={Hill, Michael A.},
    AUTHOR={Mehrle, David},
    AUTHOR={Quigley, James D.},
TITLE={Free Incomplete Tambara Functors are Almost Never Flat}, 
JOURNAL={International Mathematics Research Notices},
YEAR={2023},
ISSUE= {5},
PAGES={4225--4291},
}


\bib{lang}{book}{
    AUTHOR = {Lang, Serge},
     TITLE = {Algebra},
   EDITION = {revised third edition},
 PUBLISHER = {Springer-Verlag, New York, Grad. Texts in Math. 211},
      YEAR = {2002},
     PAGES = {xvi+914},
}
     
  \bib{leeman}{misc}{
  AUTHOR = {Leeman, Ethan Jacob},
   TITLE = {Andr\'e-Quillen (co)homology and equivariant stable homotopy
theory}, 
      YEAR = {2019},
NOTE={Dissertation, The University of Texas at Austin
}, 
    }

\bib{lewis}{inproceedings}{
AUTHOR = {Lewis, L. Gaunce},
  TITLE = {The {RO}({G})-Graded Equivariant Ordinary Cohomology of Complex Projective Spaces with Linear {$\mathbb{Z}/p$} Actions},
  booktitle = {Algebraic {Topology} and {Transformation Groups}},
  date = {1988},
  pages = {53--122},
  publisher = {Springer Berlin Heidelberg},
}
    
\bib{LewisPiF}{article}{
    AUTHOR = {Lewis, L. Gaunce},
     TITLE = {When projective does not imply flat, and other homological
              anomalies},
   JOURNAL = {Theory Appl. Categ.},
  FJOURNAL = {Theory and Applications of Categories},
    VOLUME = {5},
      YEAR = {1999},
     PAGES = {No. 9, 202--250},
}
   

    \bib{lrz}{misc}{
      AUTHOR = {Lindenstrauss, Ayelet},
      AUTHOR = {Richter, Birgit},
      AUTHOR = {Zou, Foling},
      TITLE ={Loday constructions for Tambara functors},
      NOTE = {preprint, 	arXiv:2304.01656}}

    \bib{mazur}{misc}{
      AUTHOR={Mazur, Kristen},
      TITLE={On the structure of Mackey functors and Tambara functors}, 
      YEAR = {2013}, 
      NOTE={Dissertation, University of Virginia}, 
}
    
    
\bib{strickland}{misc}{
  AUTHOR = {Strickland, Neil},
   TITLE = {Tambara functors}, 
      NOTE={preprint, arXiv:1205.2516}, 
    }

\bib{zeng}{misc}{
      title={Equivariant Eilenberg-Mac Lane spectra in cyclic $p$-groups}, 
      author={Zeng, Mingcong},
      note={preprint, arXiv:1710.01769}
}

    
\end{biblist}
\end{bibdiv}
\end{document}